\PassOptionsToPackage{dvipsnames}{xcolor}
\documentclass[11pt,reqno,oneside]{amsart}
\usepackage{amsmath, amssymb, graphicx, mathrsfs, dsfont, hyperref}
\usepackage{epsfig}
\usepackage{setspace}
\usepackage{amsthm}
\usepackage[top=1in, bottom=1in, left=1in, right=1in]{geometry} 
\usepackage{pgfplots}
\usepackage[dvipsnames]{xcolor}
\usepackage{mathtools}
\usepackage{skull, color}
\usepackage{float}
\usepackage{cite}
\usepackage{caption}
\usepackage{adjustbox}
\usepackage{enumerate}
\usepackage{qrcode}

\theoremstyle{definition}
\newtheorem{theorem}{Theorem}[section]
\newtheorem{conjecture}[theorem]{Conjecture}
\newtheorem{corollary}[theorem]{Corollary}

\newtheorem{lemma}[theorem]{Lemma}
\newtheorem{question}[theorem]{Question}

\newtheorem{remark}[theorem]{Remark}

\newcommand{\B}{\mathbf{B}}
\newcommand{\BB}[2]{\mathbf{B}_j^k}
\newcommand{\bov}[2]{\overline{\mathbf{B}_j^k}}

\renewcommand{\th}{\textsuperscript{th}}

\pgfplotsset{compat=1.18}

\begin{document}

\title[Integers that are not the Sum of Positive Powers]{Integers that are not the Sum of Positive Powers}

\author{Brennan Benfield and Oliver Lippard}
\address{Brennan Benfield: Department of Mathematics and Statistics, University of North Carolina at Charlotte, 9201 University City Blvd., Charlotte, NC 28223, USA}
\address{Oliver Lippard: Department of Mathematics and Statistics, University of North Carolina at Charlotte, 9201 University City Blvd., Charlotte, NC 28223, USA}

\thanks{The authors extend their sincerest gratitude to Dr. Sorosh Adenwalla for helpful suggestions that greatly improved the quality of this paper.}

\thispagestyle{empty}
\begin{abstract}
The generalized Waring problem asks exactly which positive integers \textit{cannot} be expressed as the sum of $j$ positive $k\textsuperscript{th}$ powers? Using computational techniques, this paper refines an approach introduced by Zenkin, establishes results for the individual cases $5 \le k \le 9$, and resolves conjectures of Zenkin and the OEIS. This paper further establishes theoretical results regarding the properties of the sets of integers that are not the sum of $j$ positive $k\textsuperscript{th}$ powers. The notion of Waring's problem is further extended to the finite sets of non-representable numbers where $G(1,k) < j < g(1,k)$. Improved computational techniques and results from Waring's problem are used throughout to catalog the sets of such integers, which are then considered in a general setting.
\end{abstract}


\maketitle

\section{Introduction}\label{Introduction}

Diophantus asked if every positive integer could be written as the sum of at most four perfect squares, which was famously proven by Lagrange \cite{Lagrange} in 1770. Waring \cite{Waring} quickly proposed the generalization: denote by $g(k)$ the least number $s$ such that every positive integer is the sum of at most $s$ positive $k\th$ powers. In this notation, Lagrange proved that $g(2)=4$. Waring initially conjectured that $g(3) \leq 9$ and $g(4) \leq 19$, and in general that $g(k) < \infty$. That $g(k)$ is finite for all $k$ was settled by Hilbert \cite{Hilbert} in 1932. Knowing there are only finitely many positive integers that are \textit{not} expressible as the sum of positive $k\th$ powers begs the question: Exactly which positive integers \textit{cannot} be written as the sum of exactly $j$ positive $k\th$ powers? 

Small values of $g(k)$ are known; notably Euler's son, J. A. Euler, \cite{Dickson4} is credited with the current lower bound: $2^k-\lfloor\left(\frac{3}{2}\right)^k\rfloor-2 \leq g(k)$ which has been verified for $6 \leq k \leq 471600000$ \cite{Kubina,Stemmler}. The conjecture is that this lower bound is actually the explicit formula for $g(k)$ and has all but been proven (see Kubina \& Wunderlich \cite{Kubina}, Mahler \cite{Mahler}, Niven \cite{Niven}, Pillai \cite{Pillai}, etc.). The sequence given by $g(k)$ as $k$ increases begins \cite{OEIS}:
\[
g(k) = 1, 4, 9, 19, 37, 73, 143, 279, \ldots\ \qquad \text{for}
\ k=1, 2, 3, \ldots\] 

Denote by $G(k)$ the least number $s$ such that every \textit{sufficiently large} positive integer is the sum of at most $s$ positive $k$\textsuperscript{th} powers. The current understanding of $G(k)$ is much less complete than that of $g(k)$. The only two known values are $G(2)=4$ (Lagrange in 1770 \cite{Lagrange}) and $G(4)=16$ (Davenport in 1939 \cite{Davenport2}), though upper and lower bounds exist for several small values:
\[
4 \leq G(3) \leq 7, \quad 6 \leq G(5) \leq 17, \quad 9 \leq G(6) \leq 24,\quad 8 \leq G(7) \leq 33,\quad 32 \leq G(8) \leq 42,
\]
\noindent et cetera. In general, explicit upper bounds for $G(k)$ have been established, first in 1919 by Hardy and Littlewood \cite{Hardy_Littlewood} who showed that $G(k) \leq (k-2)\cdot2^{k-1}+5$. 
Subsequently, the work of Vinogradov \cite{Vinogradov1934,Vinogradov1934.2,Vinogradov1934.3,Vinogradov1934.4,Vinogradov1934.5,Vinogradov1935,Vinogradov1935.2,Vinogradov1935.3,vinogradov1947,Vinogradov1959}, Tong \cite{Tong}, Chen \cite{Chen}, Karatsuba \cite{Karatsuba}, Vaughan \cite{Vaughan}, and Wooley \cite{Wooley}, to name a few, led to the recent publication in 2022 by Br\"udern \& Wooley \cite{BruedernWooley} that 
\[
G(k) \leq \lceil k(\log k + 4.20032)\rceil.
\]

Define a number to be $(j,k)$-representable if it is the sum of $j$ positive $k\th$ powers, and denote this representation a $(j,k)$-representation. Following the work of Zenkin \cite{Zenkin1, Zenkin2, Zenkin3} on the generalized Waring problem, denote by $G(1,k)$ the smallest integer such that for all $j \geq G(1,k)$, all but finitely many positive integers are $(j,k)$-representable, where the first input to the function refers to 1 being the smallest number permissible as a base. The classical Waring problem may be used to bound $G(1,k)$ and achieve related results, although it is clear there is no hope to achieve a bound for $G(1,k)$ lower than that of $G(k)$.

Let $\mathbf{B}_j^k$ denote the set of positive integers that are not the sum of \textit{exactly} $j$ positive $k\th$ powers (the numbers that are not $(j,k)$-representable). This set includes three pieces: the set of all $n<j$, a set $\mathbf{B}^k$ that is computed for each $k$, and a set $\overline{\mathbf{B}_j^k} = \{n-j : n > j, \in \mathbf{B}_j^k\} \setminus \mathbf{B}^k$.



Denote by $g(1,k)$ the smallest $j$ such that $\overline{\mathbf{B}_k^j}$ is empty. This paper determines specific values of $G(1,k)$ and $g(1,k)$, and computes the associated $\mathbf{B}$-sets for $3 \le j \le 9$. In general, a pattern emerges for all values of $k$: for small $j$, there are infinitely many positive integers that are not $(j,k)$-representable, until $j$ is large enough so that there are only finitely many positive integers are not $(j,k)$-representable (namely $\mathbf{B}^k \cup \overline{\mathbf{B}_{j}^k}$), until again $j$ is large enough at which point the set $\overline{\mathbf{B}_{j}^k}$ is empty. This phenomenon was summarized by Zenkin for all $k$.
\ref{Conjecture_1}, \ref{conjecture_2}, \& \ref{Conjecture3}.
\begin{theorem}\label{Conjecture_1} (Zenkin \cite{Zenkin3})
    For every $k \geq 2$, there exists an index $G(1,k)$ such that for all $j \geq G(1,k)$, all but finitely many positive integers are $(j,k)$-representable.
\end{theorem}
\noindent Theorem \ref{Conjecture_1} may appear obvious from Hilbert's \cite{Hilbert} result that $G(k) \leq g(k)<\infty$, but recall that $G(2)=4$ even though there are infinitely many positive integers that are not the sum of $4$ positive squares.
\begin{conjecture}\label{conjecture_2}
    For every $k \geq 2$, there exists an index $g(1,k)$ and a finite set $\mathbf{B}^k$ such that for all $j \geq g(1,k)$, every positive integer $n$ is $(j,k)$-representable, provided $n > j-1$ and $n \neq j+\beta$ for $\beta \in \mathbf{B}^k$.
\end{conjecture}
\begin{remark}\label{Conjecture3}
    For $G(1,k) \leq j < g(1,k)$, there exists a finite, nonempty set $\overline{\mathbf{B}_{j}^k}$ such that every positive integer $n$ has a $(j,k)$-representation, provided $n > j-1$ and $n \neq j+\beta$ for $\beta \in \mathbf{B}^k\cup \overline{\mathbf{B}_{j}^k}$.
\end{remark}

The proof techniques used in this area were developed for $k=2$ independently by Dubouis \cite{Dubouis} and by Pall \cite{Pall}, but were adapted and generalized by Zenkin \cite{Zenkin2} for higher powers. Following his notation, let $n^*$ denote a positive integer that is the sum of $g(k)+d, g(k)+d-1, g(k)+d-2, \ldots, d$ positive $k\th$ powers. Ideally, $d=2$, but there is heuristic evidence (see Section \ref{Improvements}) that $d\geq3$ for all $k\geq6$. In reality, computing small values of $d$ for large values of $k$ becomes unattainable, and the computation is performed for the best candidate $n^*$. Table \ref{Table_Powers} lists the results of the best known candidate $n^*$ and the corresponding results for $k\leq9$.

\begin{table}[ht!]
    \centering
    \begin{tabular}{c|c|c|c|c|c}
         $k$ & $n^*$ & $d$ & $G(k)+d$ & $G(1,k)$ & $g(1,k)$\\
         \hline
         2 & 169 & 1 & 5 & 5 \cite{Dubouis, Pall}& 6 \cite{Dubouis, Pall}\\
         \hline
         3 & 1072 & 2 & 9 & $\leq9$ \cite{Zenkin1}\ \ \ \ \ \ \ \ \ & 14 \cite{Zenkin1}\ \ \ \ \ \ \ \\
         \hline
         4& 77900162 & 2 & 18 & $\leq18$ \cite{Zenkin1}\ \ \ \ \ \ \ \ \ \ & 21 \cite{Zenkin1}\ \ \ \ \ \ \ \\
         \hline
         5& 100000497376 & 3 & 20 & $11^*$ & $57^\text{\textdagger}$\\
         \hline
         6 & 41253168892& 5 & 29 & $18^*$ & $78^{\star\star}$\\
         \hline
         7 & 822480142011 & 7 & 40 & $25^*$ & $245^{\star\star}$\\
         \hline
         8& 17373783550950 & 9 & 51 & $47^*$ & $334^{\star\star}$\\
         \hline
         9 & 25636699123453928 & 14 & 64 & $121^*$ & $717^{\star\star}$
    \end{tabular}
    \caption{$G(1,k)$ and $g(1,k)$ for $2\leq k \leq9$; $^{\star\star}$New theorem, $^*$New conjectures, $^\text{\textdagger}$Conjectured by Zenkin \cite{Zenkin3}.}
    \label{Table_Powers}
\end{table}

The sets $\overline{\mathbf{B}_{j}^k}$ are verified for all $j\geq g(k)+d$. 
For each $k=5, 6, 7, 8, 9$, direct computation predicts the numeric upper bound, and additional sets $\overline{\mathbf{B}_{j}^k}$ for $j < g(k)+d$ are conjectured. Particular indices $j$ of note are $G(1,k)$ and $g(1,k)$, at which point, the conditions of Theorem \ref{B_consistency} are satisfied and the sequence has no surprises left. Table \ref{Table_Powers} uses the proven upper bounds for $G(k)$, although for nearly all $k$, it is conjectured that the upper bound is not tight.

This paper proves Zenkin's conjecture that $g(1,k)=57$, resolves several conjectures found in the Online Encyclopedia of Integer Sequences (OEIS) \cite{OEIS}, and determines bounds for $G(1,k)$ for $4<k<9$ for sufficiently large $n$. 


\section{Computational Methods}\label{Methods}


\subsection{Existence of Representations}\label{Computation}

For most of the results in this paper, it is only necessary to check whether or not a $(j,k)$-representation \textit{exists}, rather than actually finding the representation. The main idea behind many of the computational theorems is that if $n$ is $(j,k)$-representable, then $n+a^k$ is $(j+1,k)$-representable for any positive integer $a$. This can be seen by taking a $(j,k)$-representation of $n$ and adding $a^k$ to the end. Theorem \ref{n*_application} is an extension of the computational technique established by Zenkin \cite{Zenkin1}, generalized to allow $G(k)$ to be used, in addition to $g(k)$.
\begin{theorem}\label{n*_application}
    Suppose there exists a number $n^* > 0$ that is $(j,k)$-representable for all $d \le j < b+d$, where $b$ is chosen such that all positive integers $n > N$ can be written as a sum of at most $b$ positive $k\th$ powers. Then, all positive integers greater than $N + n_0$ are $(b+d,k)$-representable.
\end{theorem}

\begin{proof}
    Let $n > n^* + N$ and consider when $j=b+d$. Since $n - n^* > N$, it can be written as 
    \[
    n-n^*=\sum_{i=1}^{b}x_i^k\ \quad \text{where}\ \ x_{b} \geq \ldots \geq x_1 \geq 0.
    \]
    Let $b'$ be the number of nonzero $x_i$'s, which must be positive, since $n^* > 0$. Then, it is possible to write the following $(b+d,k)$-representation of 
    \[
    n = \sum_{i=1}^{b'}x_i^k + \sum_{\iota=1}^{d+b-b'}y_\iota^k,
    \]
    where $\sum_{\iota=1}^{d+b-b'}y_\iota^k$ is a $(d+b-b',k)$-representation of $n^*$, which is assumed to exist by the condition in the theorem.
\end{proof}

\begin{theorem}\label{representation_existence}
    Suppose that there exist positive integers $n$ and $N$ such that for all $n < m < N$, $m$ is $(j,k)$-representable. If $a^k-(a-1)^k < N-n$, then for all $n+1 < m' < N+a^k$, $m'$ is $(j+1,k)$-representable.
\end{theorem}
\begin{proof}
    If there exists some $b$ such that $n < m'-b^k < N$, then $m'-b^k = \sum_{i=1}^j c_i^k$ by the assumption that all numbers between $n$ and $N$ $(j,k)$-representable Hence, $m'$ is $(j+1,k)$-representable as $m' = \sum_{i=1}^{j+1} c_i^k$, where $b = c_{j+1}$. The remainder of the proof is devoted to showing that such a $b$ exists. Choose $b=\min(\{b: m'-b^k < N\})$. Since $m' < N+a^k$, $m'-a^k < N$, so $b \le a$. Suppose towards a contradiction that $m'-b^k \le n$. Since $b \le a$, $b^k-(b-1)^k < a^k-(a-1)^k < N-n$. Thus, $m'-(b-1)^k \le n+b^k-(b-1)^k \le n+a^k-(a-1)^k < N$, so the original value of $b$ must not have been minimal. Therefore, every integer between $n+1$ and $N+a^k$ is $(j+1,k)$-representable.
\end{proof}

To calculate sets $\mathbf{B}_j^k$, it is necessary to determine which of a wide range of numbers are $(j,k)$-representable. This leads to the memory-intensive algorithm to determine the existence of representations. This sieve-type algorithm is based on the principle that if $n_1$ is $(j_1,k)$-representable and $n_2$ is $(j_2,k)$-representable, then $n_1+n_2$ is $(j_1+j_2,k)$-representable by adding the representations of $n_1$ and $n_2$. Setting $j_2$ to 1 simplifies the algorithm by reducing the possible values of $n_2$ to perfect $k\th$ powers. The first step is to compute whether each $n$ in the interval $0 < n < N$ is $(j,k)$-representable for some large $N$; this data can be stored as a list (or string) of 0s and 1s. Next, the list corresponding to $j+1$ is computed according to the following rule: $n$ is $(j+1,k)$-representable if and only if there exists some $a$ such that $n-a^k$ is $(j,k)$-representable; this data is assumed to be fed into the algorithm. This process is then repeated until the desired $j$ is achieved, with runtime proportional to $j$. While requiring $2N$ bits of RAM, this algorithm is very computationally efficient for computing all $(j,k$)-representable numbers less than $N$ for reasonably small $j$.

\section{Computing $n^*$}

The goal of this section is to find a value of $n^*$ that minimizes $d$ (thus reducing the bound for $G(1,k)$). While the theoretical lower limit for $d$ is 2 with all exponents $k > 2$, due to Fermat's Last Theorem, this theoretical limit is difficult and perhaps impossible to achieve. For sufficiently large $j$ and $n$, a $(j,k)$-representation is almost certain to exist; hence, to find a suitable $n^*$, the biggest challenges are the smallest values of $j$. Candidate $n^*$ values that are $(j,k)$-representable for $j=d, d+1$, and for several more $j$, are generated using a sieve-like approach. First, a value of $d$ must be chosen, as well as an interval $(n,N)$ in which to search. Then, $(d,k)$-representations that evaluate to between $n$ and $N$ are computed, resulting in a list (similar to the algorithm above) of $(d,k)$-representable numbers. Next, $(d+1,k)$-representations in the interval are computed, but only recorded if they evaluate to one of the numbers listed in the previous step. The same process is completed for $d+2$ and $d+3$, and the output is the numbers recorded in the final step. Another way to generate candidate values of $n^*$ is to start with a number $\nu$ that is $(\delta,k)$-representable and $(\delta+1,k)$-representable. Then, $n^* = 2\nu$ can be expressed as a sum of any combination of these two representations, which makes it $(2\delta,k)$-representable, $(2\delta+1,k)$-representable, and $(2\delta+2,k)$-representable; hence, $d=2\delta$. While this approach does not minimize $n^*$, it reduces the number of representations that have to be evaluated, which provides an alternate pathway if the original algorithm's runtime is excessive. This is how the $n^*$ value for ninth powers was obtained.

The result of this process is likely to be a small collection of numbers, and subsequent values of $j$ are so large that it is onerous to calculate all $(j,k)$-representations in a certain interval. This need to verify the existence of a $(j,k)$-representation for large $n$ leads to the second, computation-intensive algorithm. This algorithm, which actually computes a representation if one exists, performs a similar function to Wolfram Mathematica's \texttt{PowersRepresentations} function. The representation is assumed to be written in non-increasing order, with the largest elements assigned first. This often decreases the number of possible values that must be tested, which in turn decreases runtime. Note that the largest element must be between $\left(\frac{n}{j}\right)^{1/k}$ and $n^{1/k}$ (which are often close to each other), while the smallest element may be any value between 1 and $\left(\frac{n}{j}\right)^{1/k}$. For each choice of the largest value, a new instance of the algorithm is run with $j-1$ in place of $j$, until $j=1$, at which point either $n$ is a perfect $k\th$ power, and a representation is outputted, or it is not, and the algorithm makes another guess for the largest number in the representation. With the use of some memory, the algorithm may be further improved by computing a list of $(j',k)$-representable numbers less than $n$ for suitably small $j'$, usually 4. When an instance is opened with $j=4$, the algorithm checks if $n$ is in the list; if it is, then the algorithm continues, if not, then the instance closes and changes one of its earlier guesses. This quick memory check saves the algorithm 4 iterations when it is certain to fail.

\section{The Known Cases: Squares, Cubes, \& Fourth Powers} 
\subsection{Squares} To justify Conjectures \ref{Conjecture_1} and \ref{conjecture_2} with known results, consider the smallest case, when $k=2$. There are infinitely many positive integers that are not $(j,2)$-representable for $j=1,2,3,4$. It is well known \cite{Burton} that a positive integer is $(2,2)$-representable if and only if its prime factors congruent to 3 modulo 4 have even exponents. 

Exactly which positive integers have a $(3,2)$-representation is still open, and depends on determining all discriminants of binary, positive definite quadratic forms with exactly one class in each genus (see Chowla \cite{Chowla}), but  the work of Hurwitz \cite{Hurwitz}, Pall \cite{Pall}, and Grosswald, Calloway, \& Calloway \cite{Grosswald_C_C} showed that there is a finite set $\mathbf{T}$ such that every positive integer $n$ has a $(3,2)$-representation provided $n \neq 4^\alpha(8m+7)$ and $n=4^\alpha t$ for $t\in\mathbf{T}$ and $\alpha=0,1,2,\ldots$. The set $\mathbf{T}$ is conjectured to be $\mathbf{T} = \{1, 2, 5, 10, 13, 25, 37, 58, 85, 130\}$ \cite{Grosswald_C_C}, and Weinberger \cite{Weinberger} showed that the set $\mathbf{T}$ may contain at most one more element, which must be greater than $5\cdot10^{10}$.

Lagrange \cite{Lagrange} managed to show that $G(2)=4$ \textit{despite} the fact that there are infinitely many positive integers that are not $(4,2)$-representable, a fact first conjectured in 1638 by Descartes \cite{Descartes} and proven in 1911 by Dubouis \cite{Dubouis}:
\[
\text{OEIS A000534 \cite{OEIS}}: 1, 2, 3, 5, 6, 8, 9, 11, 14, 17, 24, 29, 32, 41, 56, 96, 128, 224, \ldots
\]
Lagrange achieved his proof only because every member of this list is a sum of $1$ or $2$ or $3$ positive squares. Grosswald \cite{Grosswald}, elegantly summarizes the result by introducing the set: \[\mathbf{B}^2=\{1, 2, 4, 5, 7, 10, 13\}.\]
Dubouis showed that every positive integer $n$ is $(4,2)$-representable, provided $n \neq 1, 2, 3$ and $n \neq 4+b$ for $\beta\in\mathbf{B}^2\cup\{25,37\}$ and $n \neq 2*4^\alpha, 6*4^\alpha, 14*4^\alpha$ for $\alpha=0,1,2,\ldots$.

 For $j \geq 5$, the situation is much different - there are only \textit{finitely} many integers that are not $(5,2)$-representable. This seems like a consequence of Hilbert's \cite{Hilbert} result, that $G(k) \leq g(k)<\infty$, but recall there are infinitely many positive integers that are not $(4,2)$-representable, yet $G(2)=4$. Determining precisely which integers are not $(5,2)$-representable was accomplished by Dubouis \cite{Dubouis}, Niven \cite{Niven2,Niven3} and Pall \cite{Pall}. Every positive integer $n$ is $(5,2)$-representable provided $n>j-1$ and $n \neq j+\beta$ for $\beta\in\mathbf{B}^2 \cup \{28\}$.

 Finally, for all $j\geq6$, Dubouis \cite{Dubouis}, Niven \cite{Niven2,Niven3} and Pall \cite{Pall} showed that every positive integer $n$ is $(j,2)$-representable provided $n>j-1$ and $n \neq j+\beta$ for $\beta\in\mathbf{B}^2$.

\subsection{Cubes}
What positive integers cannot be written as the sum of exactly $j$ positive cubes? For $k=3$, Zenkin \cite{Zenkin1} was the first to compute the set $\mathbf{B}^3$:
\begin{align*}
\mathbf{B}^3=\{&1, 2, 3, 4, 5, 6, 8, 9, 10, 11, 12, 13, 15, 16, 17, 18, 19, 20, 22, 23, 24, 25, 27,\\ 
& 29, 30, 31, 32, 34, 36, 37, 38, 39, 41, 43, 44, 45, 46, 48, 50, 51, 53, 55, 57,\\
& 58, 60, 62, 64, 65, 67, 69, 71, 72, 74, 76, 79, 81, 83, 86, 88, 90, 93, 95, 97,\\
&100, 102, 107, 109, 114, 116, 121, 123, 128, 135, 142, 149\}.
\end{align*}
Zenkin goes on to prove that $G(1,3)\leq9$ and that $g(1,3)=14$ but does not compute the corresponding sets $\overline{\mathbf{B}_{j}^3}$ for $j<14$. In fact, the sets $\mathbf{B}_{11}^3,\ \mathbf{B}_{10}^3,\ \text{and}\  \mathbf{B}_{9}^3$ are conjectured in the Online Encyclopedia of Integer Sequences (OEIS) \cite{OEIS}, sequences A332111 \cite{OEIS}, A332110 \cite{OEIS}, and A332109 \cite{OEIS} respectively, which have now been determined with certainty. 
\begin{theorem}\label{theorem_3_cubes}
    For $j = 13, 12, 11, 10$ and $9$, every positive integer $n$ is the sum of $j$ positive cubes provided $n > j-1$ and $n \neq j+\beta$ for $\beta \in \mathbf{B}_j^3 \cup \overline{\mathbf{B}_{j}^3}$ where:
    \begin{align*}
        \overline{\mathbf{B}_{13}^3} = \{&212\}\\
        \overline{\mathbf{B}_{12}^3} = \{&186, 205, 212\}\\
        \overline{\mathbf{B}_{11}^3} = \{&160, 179, 186, 198, 205, 212, 310\}\\
        \overline{\mathbf{B}_{10}^3} = \{&153, 160, 172, 179, 186, 191, 198, 205, 212, 247, 284, 303, 310, 364\}\\ 
        \overline{\mathbf{B}_{9}^3} = \{&153, 160, 165, 172, 179, 184, 186, 191, 198, 205, 212, 221, 238, 240, \\
        &247, 258, 277, 284, 296, 301, 303, 310, 338, 357, 364, 413, 462\}.
    \end{align*}
\end{theorem}
Also found in the OEIS, the sequence A332107 lists the conjectured set $\mathbf{B}_7^3$, A057907 lists the conjectured set $\mathbf{B}_6^3$, A057906 \cite{OEIS} lists the set $\mathbf{B}_5^3$ related to a conjecture of Romani \cite{Romani} and the sequence A057905 \cite{OEIS} lists part of the set $\mathbf{B}_4^3$. These sets are becoming too large to display here, but the ones of reasonable size will be included for completeness in the \nameref{Appendix}. 

The sets $\mathbf
 {B}_3^3$ and $\mathbf
 {B}_2^3$ (and of course $\mathbf
 {B}_1^3$) of positive integers that cannot be expressed as the sum of three or two (or one) positive cubes are known to be infinite:  Davenport \cite{Davenport} showed that every positive integer congruent to $4$ or $5$ modulo $9$ cannot be expressed as the sum of three positive cubes, 
 Euler \cite{Euler} proved the particular case of Fermat's last theorem, that no cube is the sum of two positive cubes. 

\subsection{Fourth Powers}
What positive integers are not the sum of exactly $j$ positive fourth powers? Hardy and Littlewood \cite{Hardy_Littlewood} showed that $G(4) \leq 19$ and shortly afterwards Davenport \cite{Davenport2} proved that $G(4)=16$. Numerical studies by Deshouillers, Hennecart, Kawada, Landreau, and Wooley \cite{Kawada_Wooley,Deshouillers2000,Deshouillers_Kawale_Wooley} collectively determined the $\mathbf{B}_{16}^4$. \begin{lemma}[Deshouillers, Hennecart, Kawada, Landreau, Wooley]\label{lemma_for_fourths}
Every positive integer greater than $13792$ can be expressed as a sum of at most 16 nonnegative fourth powers (which they call biquadrates). 
\end{lemma}
\noindent Zenkin \cite{Zenkin3} used the fact that $G(4)=16$ to prove $G(1,4) \le 18$. However, he was unable to compute the set $\overline{\mathbf{B}^4_{18}}$because Lemma \ref{lemma_for_fourths} had not yet been proved, though he did determine the set
\[
\mathbf{B}^4 = \{ 1, 2, 3, 4, 5, 6, \ldots 2561, 2566, 2581, 2596, 2611, 2626, 2641 \}.
\]
\begin{theorem}[Zenkin \cite{Zenkin3}]\label{Theorem_2_fourths} 
    Every positive integer $n$ is the sum of $j \ge 21$ positive fourth powers provided $n > j-1$ and $n \neq j + b$ for $b \in \mathbf{B}^4$.
\end{theorem}
\noindent The set of positive integers that are not the sum of $j=19$ \textit{nonnegative} fourth powers was determined in 1992 by Deshouillers and Dress \cite{Deshouillers1992.2}, and in 2000, Deshouillers, Hennecart, \& Landreau \cite{Deshouillers2000} determined similar sets for $j=18, 17$ and $16$. 
\begin{theorem}\label{theorem_fourths}
    The sets $\overline{\mathbf{B}_{20}^4}$, $\overline{\mathbf{B}_{19}^4}$, and $\overline{\mathbf{B}_{18}^4}$ are finite and listed in the \nameref{Appendix}.
\end{theorem}
\begin{conjecture}
    The sets $\overline{\mathbf{B}_{17}^4}$ and $\overline{\mathbf{B}_{16}^4}$ are finite and listed in the \nameref{Appendix}.
\end{conjecture}
\noindent It is known \cite{Kawada_Wooley} that the sets $\mathbf{B}_j^4$ for $j \leq 15$ are infinite. 

\section{Positive Fifth Powers}
The landscape for higher powers is much less known. In the case of fifths (and for all larger cases), the results of Waring's problem are currently established for \textit{sufficiently large} $n$. Numerical experimentation indicates the smallest such positive integer that is not the sum of exactly $j$ positive $k$\textsuperscript{th} powers, but this $n$ is usually conjectured because the proof technique laid out by Dubouis \cite{Dubouis} and that is used in this paper has its limitations - namely, the technique requires finding an integer that is the sum of $a, a+1, a+2, \ldots, g(k)+a$ perfect $k$\textsuperscript{th} powers, creating an upper bound, and then checking small values less than that bound. Finding such a number is usually possible (though large), and proving the novel theorems in this section for sufficiently large positive integers is still progress, but ultimately a calculation will need to be performed to remove any doubts of small-valued counterexamples.


What positive integers cannot be expressed as the sum of exactly $j$ positive fifth powers? Vaughan \cite{Vaughan} established the upper bound for nonnegative fifth powers $G(5) \leq 19$, which was then refined by Br\"udern \cite{Bruedern} to $G(5) \leq 18$. Vaughan \& Wooley \cite{Vaughan_Wooley_5ths} eventually reduced the upper bound to its current status, $G(5) \leq 17$.
\begin{lemma}[Vaughan \& Wooley]\label{lemma_for_fifths}
Every sufficiently large positive integer can be expressed as a sum of at most 17 positive fifths. 
\end{lemma}
\noindent Note that Vaughan \& Wooley \cite{Vaughan_Wooley_5ths} do not calculate the least positive integer such that all larger positive integers can be expressed as the sum of at most 17 positive fifth powers.
\begin{conjecture}\label{Conjecture_n>87918 for fifths}
    Every positive integer greater than $87918$ is the sum of at most $17$ positive fifth powers.
\end{conjecture}
Define the set $\mathbf{B}^5$ by
\[
\mathbf{B}^5 = \{1, 2, 3, 4, 5, 6, 7, \ldots, , 6137, 6168, 6199, 6230, 6261\}.
\]
The full set $\mathbf{B}^5$ has $3175$ elements and is listed in the \nameref{Appendix}. Because Vaughan \& Wooley's result \cite{Vaughan_Wooley_5ths} is established for sufficiently large $n$ rather than for all $n > 87918$, the proofs for fifths (and subsequently for all larger powers) is slightly different.

\begin{theorem}\label{Theorem_2_fifths}
Every (sufficiently large) positive integer is the sum of $j \ge 57$ positive fifth powers provided $n > j - 1$ and $n \neq j + b$ for $b \in \mathbf{B}^5$.
\end{theorem}
\begin{theorem}\label{Theorem_3_fifths}
Every (sufficiently large) positive integer is the sum of $j$ positive fifth powers for $20 \leq j \leq 56$, provided $n > j - 1$ and $n \neq j + b$ for $b \in \mathbf{B}^5 \cup \overline{\mathbf{B}_j^5}$ where $\overline{\mathbf{B}_j^5}$ is a finite set given in the \nameref{Appendix}.
\end{theorem}

The sets $\overline{\mathbf{B}_j^5}$ for $6\leq j \leq19$ are also conjectured to be finite, but are too large to be computed completely, and are not included in the \nameref{Appendix} at this time. The sets $\overline{\mathbf{B}_j^5}$ for $1\leq j \leq5$ are known to be infinite.

\section{Positive Sixth Powers}
Vaughan \& Wooley \cite{Vaughan_Wooley_5ths} gave the upper bound for $G(6)$. Again in the case of sixth powers (and for all larger cases), the results of Waring's problem are established for \textit{sufficiently large} $n$. Numerical experimentation indicates the smallest such number.
\begin{lemma}[Vaughan \& Wooley]\label{lemma_for_sixths}
Every sufficiently large positive integer can be expressed as a sum of at most 24 positive sixth powers. 
\end{lemma}

\begin{conjecture}\label{smallest n conj for 6ths}
    Every positive integer greater than $1414564$ is the sum of at most $24$ positive sixth powers.
\end{conjecture}
Define the set $\mathbf{B}^6$ by
\[
\mathbf{B}^6 = \{1, 2, 3, 4, 5, 6, 7, \ldots, 711586, 711649\}.
\]
The full set $\mathbf{B}^6$ has $355825$ elements and is listed in the \nameref{Appendix}.

\begin{theorem}\label{Theorem_2_sixths}
Every positive integer $n$ is the sum of $j \ge 79$ positive sixth powers provided $n > j - 1$ and $n \neq j + b$ for $b \in \mathbf{B}^6$.
\end{theorem}
\begin{conjecture}\label{Conjecture_sixths}
Every positive integer $n$ is the sum of $j = 78$ positive sixth powers provided $n > j - 1$ and $n \neq j + b$ for $b \in \mathbf{B}^6$.
\end{conjecture}
\begin{conjecture}\label{Theorem_3_sixths}
Every positive integer $n$ is the sum of $j$ positive sixth powers for $8 \leq j \leq 78$, provided $n > j - 1$ and $n \neq j + b$ for $b \in \mathbf{B}^6 \cup \overline{\mathbf{B}_j^6}$ where $\overline{\mathbf{B}_j^6}$ is a finite set, conjectured in the \nameref{Appendix}\footnote{The sets $\overline{\mathbf{B}_j^6}$ for $9\leq j \leq18$ are also conjectured to be finite, but are too large to be computed completely, and are not included in the \nameref{Appendix}.}
\end{conjecture}
The sets $\overline{\mathbf{B}_j^6}$ for $1\leq j \leq8$ are known to be infinite.


\section{Positive Seventh Powers}
What positive integers cannot be expressed as the sum of exactly $j$ positive seventh powers? Vaughan \& Wooley \cite{Vaughan_Wooley_5ths} gave the current upper bound for $G(7)$.
\begin{lemma}[Vaughan \& Wooley]\label{lemma_for_sevenths}
Every sufficiently large positive integer can be expressed as a sum of at most $33$ positive seventh powers. 
\end{lemma}
\noindent Again in the case of seventh powers, the results of Waring's problem are established for \textit{sufficiently large} $n$. Numerical experimentation indicates the smallest such number.
\begin{conjecture}\label{conj for 7ths}
    Every positive integer greater than $9930770$ is the sum of $33$ positive seventh powers.
\end{conjecture}

Define the set $\mathbf{B}^7$ by
\[
\mathbf{B}^7 = \{1, 2, 3, 4, 5, 6, 7, \ldots, 248950, 249077\}.
\]
The full set $\mathbf{B}^7$ has $127839$ elements and is listed in the \nameref{Appendix}.

\begin{theorem}\label{Theorem_2_sevenths}
Every positive integer $n$ is the sum of $j \geq 245$ positive seventh powers provided $n > j - 1$ and $n \neq j + b$ for $b \in \mathbf{B}^7$.
\end{theorem}
\begin{theorem}\label{Theorem_3_sevenths}
Every positive integer $n$ is the sum of $j$ positive seventh powers for $244 \geq j \geq 51$, provided $n > j - 1$ and $n \neq j + b$ for $b \in \mathbf{B}^7 \cup \overline{\mathbf{B}_j^7}$ where $\overline{\mathbf{B}_j^7}$ is a finite set, listed in the \nameref{Appendix}.
\end{theorem}

\begin{conjecture}\label{Conjecture_sevenths}
Every positive integer $n$ is the sum of $j$ positive seventh powers for $50 \geq j \geq 25$, provided $n > j - 1$ and $n \neq j + b$ for $b \in \mathbf{B}^7 \cup \overline{\mathbf{B}_j^7}$ where $\overline{\mathbf{B}_j^7}$ is a finite set, conjectured in the \nameref{Appendix}.
\end{conjecture}
The sets $\overline{\mathbf{B}_j^7}$ for $24\leq j \leq8$ are also conjectured to be finite, but are too large to be computed completely, and are not included in the \nameref{Appendix}. The sets $\overline{\mathbf{B}_j^7}$ for $1\leq j \leq7$ are known to be infinite.

\section{Positive Eighth Powers}
What positive integers cannot be expressed as the sum of exactly $j$ positive eighth powers? Vaughan \& Wooley \cite{Vaughan_Wooley_5ths} gave the current upper bound for $G(8)$.
\begin{lemma}[Vaughan \& Wooley]\label{lemma_for_eighths}
Every sufficiently large positive integer can be expressed as a sum of at most $42$ positive eighth powers. 
\end{lemma}
\noindent Again in the case of eighth powers, the results of Waring's problem are established for \textit{sufficiently large} $n$. Numerical experimentation indicates the smallest such number.
\begin{conjecture}\label{smallest n conj for 8ths}
    Every positive integer greater than $858367748$ is the sum of at most $42$ positive eighth powers.
\end{conjecture}


Define the set $\mathbf{B}^8$ by
\[
\mathbf{B}^8 = \{1, 2, 3, 4, 5, 6, 7, \ldots, 1889986, 1890241\}.
\]
The full set $\mathbf{B}^8$ has $945121$ elements and is listed in the \nameref{Appendix}.

\begin{theorem}\label{Theorem_2_eighths}
Every positive integer $n$ is the sum of $j \ge 334$ positive eighth powers provided $n > j - 1$ and $n \neq j + b$ for $b \in \mathbf{B}^8$.
\end{theorem}

\begin{theorem}\label{Theorem_3_eighths}
    Every positive integer $n$ is the sum of $j$ positive eighth powers for $120 \le j \le 333$, provided $n > j - 1$ and $n \neq j + b$ for $b \in \mathbf{B}^8 \cup \overline{\mathbf{B}^8_j}$ where $\overline{\mathbf{B}^8_j}$ is a finite set, conjectured in the \nameref{Appendix}
\end{theorem}
For $64\leq j \leq 119$, the sets $\overline{\mathbf{B}^8_j}$ are currently too large to be computed completely, though the proof that they are finite can be obtained for \textit{sufficiently large} positive integers.
\begin{theorem}
    Every sufficiently large positive integer $n$ is the sum of $j$ positive eighth powers for $64 \le j \le 119$, provided $n > j - 1$ and $n \neq j + b$ for $b \in \mathbf{B}^8 \cup \overline{\mathbf{B}^8_j}$ where $\overline{\mathbf{B}^8_j}$ is a finite set.
\end{theorem}

The sets $\overline{\mathbf{B}_j^8}$ for $63\leq j \leq32$ are also conjectured to be finite, but are too large to be computed completely, and are not included in the \nameref{Appendix}. The sets $\overline{\mathbf{B}_j^8}$ for $1\leq j \leq32$ are known to be infinite.

\section{Positive Ninth Powers}
What positive integers cannot be expressed as the sum of exactly $j$ positive ninth powers? Vaughan \& Wooley \cite{Vaughan_Wooley_5ths} gave the current upper bound for $G(9)$.
\begin{lemma}[Vaughan \& Wooley]\label{lemma_for_ninths}
Every sufficiently large positive integer can be expressed as a sum of at most $50$ positive ninth powers. 
\end{lemma}
\noindent Again in the case of ninth powers, the results of Waring's problem are established for \textit{sufficiently large} $n$. Unlike the previous cases, computation to establish the smallest such number is out of reach for the modest computers the authors are using.
\begin{question}\label{smallest n conj for 9ths}
    What is the largest positive integer that is \textit{not} the sum of at most $50$ positive ninth powers?
\end{question}

The nature of the proof technique used in the later sections requires an exact value for the largest such number, but in the case of ninths, is still unknown. Without the smallest such number, the proof cannot be accomplished, yet direct calculation yields the likely result. Define the set $\mathbf{B}^9$ by
\[
\mathbf{B}^9 = \{1, 2, 3, 4, 5, 6, 7, \ldots, 6463886, 6464397\}.
\]
The full set $\mathbf{B}^9$ has $3438463$ elements and is listed in the \nameref{Appendix}.

\begin{conjecture}\label{Theorem_2_ninths}
Every positive integer $n$ is the sum of $j \geq 717$ positive ninth powers provided $n > j - 1$ and $n \neq j + b$ for $b \in \mathbf{B}^9$.
\end{conjecture}

\begin{conjecture}\label{Theorem_3_ninths}
    Every positive integer $n$ is the sum of $j$ positive ninth powers for $13 \le j \le 716$, provided $n > j - 1$ and $n \neq j + b$ for $b \in \mathbf{B}^9 \cup \overline{\mathbf{B}^9_j}$ where $\overline{\mathbf{B}^9_j}$ is a finite set, conjectured in the \nameref{Appendix}\footnote{The sets $\overline{\mathbf{B}_j^9}$ for $120\leq j \leq13$ are conjectured to be finite, but are too large to be computed completely, and are not included in the \nameref{Appendix}.}
\end{conjecture}

The sets $\overline{\mathbf{B}_j^9}$ for $1\leq j \leq12$ are known to be infinite.






\section{General Properties of B-sets}\label{Section_properties_sets}
There are many properties and relations among the sets of positive integers that cannot be written as the sum of $j$ positive $k$\th powers. Natural question to ask are, how quickly do the sets $\mathbf{B}^k$ grow as $k$ gets large? What is the relationship between $\overline{\mathbf{B}_j^k}$ and $\overline{\mathbf{B}_{j-1}^k}$ for a fixed $k$? How large $j$ needs to be for a fixed $k$ so that the sets $\overline{\mathbf{B}_j^k}$ are empty? What is the density of positive integers that can be expressed as a sum of $j$ positive $k$\th powers (for small $j$)?

\subsection{The size of the set $\mathbf{B}^k$ as $k$ gets large}
Consider the sequence created by the number of elements in $\mathbf{B}^k$ for $k=2,3,\ldots$ Define the sequence $a_k$ to be the maximum of $\mathbf{B}^k$ for $k \ge 2$, and define the sequence $b_k = \mid \mathbf{B}^k \mid$. It is curious that in many cases, but not all, the largest element in $\mathbf{B}^k$ is twice and one less than the size of $\mathbf{B}^k$.
\begin{table}[ht!]
    \centering
    \begin{tabular}{c|cccccccc}
         $k$& 2 & 3 & 4 & 5 & 6 & 7 & 8 & 9\\
         \hline
        $a_k$ & 13 & 149 & 2641 & 6261 & 711649 & 249077 & 1890241 & 6464397\\
        $b_k$ & 7 & 75 & 1321 & 3175 & 355825 & 127839 & 945121 & 3438463
    \end{tabular}
    \caption{$a_k$ and $b_k$ for $2 \leq k \leq 9$}
    \label{Table_Bset_size}
\end{table}

\begin{conjecture}\label{sizemax}
    For all $k, a_{2^k} = 2b_{2^k} - 1$.
\end{conjecture}
\begin{conjecture}\label{reverse}
    $n \in \mathbf{B}^{2^k}$ if and only if $a_{2^k}-n \not\in \mathbf{B}^{2^k}$.
\end{conjecture}
\begin{conjecture}\label{limitsizemax}
    $\lim_{k\rightarrow\infty} \frac{a_k}{b_k} = 2$.
\end{conjecture}
\begin{conjecture}\label{odd}
    $a_k$ and $b_k$ are odd for all $k$.
\end{conjecture}

\begin{conjecture}
    The sequence $\mid \mathbf{B}^k \mid$ for $k=2, 3, \ldots$ never becomes monotone.
\end{conjecture}

\noindent Note that Conjecture \ref{reverse} implies Conjecture \ref{sizemax}. There is a curious phenomenon regarding the size of the largest set $\overline{\mathbf{B}_j^k}$ that is nonempty - when $j=g(1,k)-1$, the set $\overline{\mathbf{B}_j^k}$ only ever contains one element, except when $k=4$.
\begin{conjecture}
    For all $k\neq4$, the set $\overline{\mathbf{B}_j^{g(1,k)-1}}$ contains exactly one element.
\end{conjecture}

\begin{theorem}
    For any odd prime $p, \mid \mathbf{B}^{p-1} \mid \ge p^{p-1}\left(\frac{p-1}{2}\right)$.
\end{theorem}

\begin{proof}
    The number of integers $n>k$ such that there are no solutions to equation \begin{align*}\label{eq4.6}
    \sum_{i=1}^j a_i^{p-1} = n
    \end{align*} is by definition the size of the set $\mathbf{B}^{p-1}$. Recall Fermat's Little Theorem: unless $p | a_i$, it must be that $a_i^{p-1} \equiv 1 \pmod p$ for any odd prime $p$. For $n < rp^{p-1}$, at most $r-1$ of the $a_i$'s can be multiples of $p$, so the sum will be in the same equivalence class modulo $p$ as at least one of the integers in the interval $[k-r+1, k]$. This leaves $p-r$ equivalence classes (modulo $p$) that cannot be expressed as a sum of exactly $j$ positive $p-1$ powers with a sum of less than $rp^{p-1}$.\\

    There are $p^{p-2}$ positive integers between $(r-1)p^{p-1}$ and $rp^{p-1}$ in each equivalence class modulo $p$. Hence, there are at least $(p-r)p^{p-2}$ integers in the interval $((r-1)p^{p-1}, rp^{p-1}]$ that cannot be expressed as a sum of $j$ positive $p-1^\textsuperscript{th}$ powers. Summing over all $r$ such that $1 \le r \le p$ gives a lower bound for the size of $\mathbf{B}^{p-1}$:
    \[ \mid \mathbf{B}^{p-1} \mid\ \ge \sum_{r=1}^{p} (p-r)p^{p-2}.\]
    Substituting $s=p-r$ gives
    \[\mid \mathbf{B}^{p-1} \mid\ \geq \sum_{s=0}^{p-1} sp^{p-2} = \frac{p(p-1)}{2}\cdot p^{p-2} = p^{p-1}\left(\frac{p-1}{2}\right).\]
\end{proof}
\begin{corollary}
    $\mid \mathbf{B}^{10} \mid \ge 129687123005$.
\end{corollary}


\subsection{Chain of Inclusion}\label{Chain Section}
Using the idea that the sets $\overline{\mathbf{B}_j^k}$ form a chain of inclusion, with $\{n-(j+1) : n > (j+1), \in \mathbf{B}_{j+1}^k\} \subseteq \{n-j : n > j, \in \mathbf{B}_j^k\}$, it is possible to show that the sets $\{n-j : n > j, \in \mathbf{B}_j^k\}$ eventually become constant. At this point, the set $\{n-j : n > j, \in \mathbf{B}_j^k\}$ is defined to be $\mathbf{B}^k$, so that the sets $\overline{\mathbf{B}_j^k} = \{n-j : n > j, \in \mathbf{B}_j^k\} \setminus \mathbf{B}^k$ are empty.

\begin{theorem}\label{subsets}
    For all positive integers $j$ and $k$, $\{n-(j+1) : n > (j+1), \in \mathbf{B}_{j+1}^k\} \subseteq \{n-j : n > j, \in \mathbf{B}_j^k\}$.
\end{theorem}

The proof of Theorem \ref{subsets} utilizes the \textit{Theory of Partitions}. The partition function $p(n)$ enumerates the number of partitions of a positive integer $n$ where the partitions are positive integer sequences $\lambda=(\lambda_1,\lambda_2,...)$ with $\lambda_1\geq\lambda_2\geq~\dots>0$ and $\sum_{j\geq1}{\lambda_j}=n$. For example, $p(4)=5$ since the only ways to partition $4$ are $4$, $3+1$, $2+2$, $2+1+1$, and $1+1+1+1$. 

It has become increasingly popular to consider \textit{restricted} partition functions, typically denoted $p_A(n)$ where $A$ is some constraint on $\lambda$. For instance, the integer $4$ could be partitioned using only prime numbers, only positive squares, only powers of $2$, only odd numbers, etc:
\begin{alignat*}
4  & 4= 2+2  &\ \ \qquad  4 &  = 2^2            &\ \ \qquad  4  & = 2^2  &\ \ \qquad  4  & = 3+1\\
   &        &          & = 1^2+1^2+1^2+1^2 &           & = 2^1+2^1  &                       &=1+1+1+1\\
  &        &          &                      &           & = 2^0+2^0+2^0+2^0  &         
\end{alignat*}
There are two restricted functions that are of particular interest regarding generalized taxicab numbers: The $k$\textsuperscript{th} power partition function and the partitions into exactly $j$ parts function. 

The $k$-th power partition function (see \cite{Benfield,Gafni,Vaughan2,Wright}), denoted $p^k(n)$, counts the number of ways that a positive integer $n$ can be written as the sum of positive $k$-th powers. Note that $p^k(n)=p(n)$ when $k=1$, and for $k=2$, $p^2(n)$ restricts $\lambda$ to positive squares. For example, $p^2(4)=2$ since $4=2^2=1^2+1^2+1^2+1^2$ and $p^k(4)=1$ for all $k\geq 3$. 

The partitions into exactly $j$ parts function (see \cite{ahlgren,kim}), denoted $p(n,j)$ counts the number of ways that a positive integer $n$ can be written as a sum of exactly $j$ integers. For example, $p(4,2)=2$ because there are exactly two ways to write $4$ as a sum of $2$ integers: $4=3+1=2+2$. 

Synthesizing the combinatorial properties of these two restricted partition functions creates the \textit{$k$\textsuperscript{th} power partition into exactly $j$ parts} function, denoted $p^k(n,j)$. In a previous paper, the authors along with A. Roy \cite{Benfield2} established a useful identity for this partition function:
\begin{lemma} \label{identity_1}
For all $n, j \ge 2$, 
\[p^k(n,j) \leq p^k(n+1, j+1)
\]
with equality whenever $n < 2^kj$.
\end{lemma}

\begin{proof}[Proof of Theorem \ref{subsets}]
The set $\overline{\mathbf{B}_j^k}$ is the set of positive integers $b$ such that $b-j$ has zero representations as a sum of $j$ positive $k^\th$ powers. For some positive integer $c$, suppose $c \not \in \overline{\mathbf{B}_j^k}$. If $p^k(c+j,j) = 0$, then $c \in \overline{\mathbf{B}_j^k}$, contradicting the hypothesis. Thus, $p^k(c+j,j) \ge 1$ for all $c$. It follows from Lemma \ref{identity_1} that $p^k(c+(j+1), j+1) \ge 1$, hence $c \not\in \overline{\mathbf{B}_{j+1}^k}$. Then, if $c \in \overline{\mathbf{B}_{j+1}^k}$, then $c \in \overline{\mathbf{B}_j^k}$.
\end{proof}

\begin{corollary}\label{subsets_overline}
    For all positive integers $j$ and $k$, $\overline{\mathbf{B}_j^k} \subseteq \overline{\mathbf{B}_j^{k+1}}$.
\end{corollary}

\begin{theorem}\label{equality}
    Let $a_{j,k} = \max \left(\mathbf{B}_j^k\right)$. If $a_{j,k} < j(2^k-1)$, then $\{n-(j+1) : n > (j+1), \in \mathbf{B}_{j+1}^k\} = \{n-j : n > j, \in \mathbf{B}_j^k\}$.
\end{theorem}
\begin{proof}
    By Theorem \ref{subsets}, $\{n-(j+1) : n > (j+1), \in \mathbf{B}_{j+1}^k\} \subseteq \{n-j : n > j, \in \mathbf{B}_j^k\}$. It remains to show that $\{n-j : n > j, \in \mathbf{B}_{j}^k\} \subseteq \{n-(j+1) : n > (j+1), \in \mathbf{B}_{j+1}^k\}$. Suppose $n \in \{n-j : n > j, \in \mathbf{B}_{j}^k\}$. Then, $p^k(n+j,j) = 0$. Because $n < j(2^k-1)$, it follows that $n+j < j(2^k)$. By the equality part of Lemma \ref{identity_1}, $p^k(n+j,j) = p^k(n+(j+1),j+1) = 0$. Hence, $n \in \{n-j : n > j, \in \mathbf{B}_{j+1}^k\}$.
\end{proof}
A useful corollary follows: it is only necessary to check up to $j = \frac{a_{j,k}}{2^k-1}$ to be sure
that a candidate set is in fact the complete set $\mathbf{B}^k$. While $a_{j,k}$ varies with $j$, $a_{j+1,k}=a_{j,k}+1$ for sufficiently large $j$, so the variation is often small.
\begin{corollary}
    If $j_0 \geq \frac{a_{j_0,k}}{2^k-1}$, then the sets $\overline{\mathbf{B}_j^k}=\emptyset$ for $j \ge j_0$.
\end{corollary}

\begin{proof}
    This can be shown inductively, with $j_0$ as a base case. Since $a_{j_0,k} < j_0(2^k-1)$, the set $\{n-(j_0+1) : n > (j_0+1), \in \mathbf{B}_{j_0+1}^k\} \subseteq \{n-j_0 : n > j_0, \in \mathbf{B}_{j_0}^k\}$. Hence, $a_{j_0+1,k}=a_{j_0,k}+1 < j_0(2^k-1)+1 \le (j_0+1)(2^k-1)$ for $k \ge 1$. Additionally, by Theorem \ref{equality}, $\{n-(j_0+1) : n > (j_0+1), \in \mathbf{B}_{j_0+1}^k\} \subseteq \{n-j_0 : n > j_0, \in \mathbf{B}_{j_0}^k\}$. The conditions of the theorem are still satisfied for $j_0+1$, which completes the inductive step. Therefore, all the sets $\{n-j : n > j, \in \mathbf{B}_{j}^k\}$ for $j \ge j_0$ are identical, so by definition the sets $\overline{\mathbf{B}_j^k}$ are empty.
\end{proof}

\noindent Note that if $\overline{\B_j^k} = \emptyset$, then $\overline{\B_{j+i}^k} = \emptyset$ for $i \ge 0$ by the inclusion property. 

\subsection{Chain of Equality}
\begin{lemma}\label{B_consistency_step1} If $n$ is a positive integer, then $n \not\in \mathbf{B}^k_{j+1}$ if and only if there is no positive $b$ such that $n-b^k \not\in \mathbf{B}^k_j$.
\end{lemma}
\begin{proof}
    Let $n$ be a positive integer; by definition, $n \in \mathbf{B}^k_j$ if and only if $n$ is not $(j,k)$-representable. Suppose $n \not\in \mathbf{B}^k_{j+1}$. Then, $n$ has a $(j+1,k)$-representation, which can be split into one $k\th$ power and a sum of $j$ positive $k\th$ powers. Hence, there exists some $b$ such that $n-b$ yields some positive integer that is $(j,k)$-representable, that is, a positive integer that is not in the set $\mathbf{B}^k_j$. To prove the converse, suppose $n \in \mathbf{B}^k_{j+1}$. If there is a  positive $b$ such that $n-b^k \not\in\mathbf{B}^k_j$, then adding $b^k$ yields a $(j+1,k)$-representation of $n$. In other words, $n\in\mathbf{B}^k_j$, which contradicts the hypothesis. Hence, there is no positive $b$ such that $n-b^k \not\in\mathbf{B}^k_j$. 
\end{proof}
\begin{lemma}\label{B_consistency_step2}
    If $\max\left(\mathbf{B}^k_j\right) < 2^k(j+1)$, then $\overline{\mathbf{B}_k^j} = \emptyset$.
\end{lemma}
\begin{proof}
    This is a consequence of Theorem \ref{equality} under particular initial conditions.
\end{proof}
\begin{lemma}\label{B_consistency_step3}
    Let $m = \max\left(\mathbf{B}^k_{j}\right)$. Suppose $\mathbf{B}^k_{j+1} = $ \{1\} $\cup$ \{$n + 1 : n \in \mathbf{B}^k_j$\} and $\left\lfloor\sqrt[k]{m+1}\right\rfloor = \left\lfloor\sqrt[k]{m}\right\rfloor$. Then, $\mathbf{B}^k_{j+2} = $ \{1\} $\cup$ \{$n + 1 : n \in \mathbf{B}^k_{j+1}$\}.
\end{lemma}
\begin{proof}
    By Lemma \ref{B_consistency_step1}, it is only necessary to check values of $b$ where $b^k < n$, since a negative number can never have a representation as a sum of positive powers. Hence, the maximum value of $b$ for which Lemma \ref{B_consistency_step1} applies is $\lfloor \sqrt[k]{m} \rfloor$. Suppose $a \in \mathbf{B}^k_j$. Then, $a+1 \in \mathbf{B}^k_{j+1}$ by assumption, and by Lemma \ref{B_consistency_step1} there exists some positive integer $1 \le b \le \lfloor \sqrt[k]{m} \rfloor$, such that $a+1-b^k \not \in \mathbf{B}^k_j$. By assumption, $a+1-b^k \not \in \mathbf{B}^k_{j+1}$, implying $a+2 \in \mathbf{B}^k_{j+2}$. On the contrary, if $a \not \in \mathbf{B}^k_j$, then there is no positive $b$ such that $a-b^k \not \in \mathbf{B}^k_j$. Because $b$ is arbitrary, $1 \le b \le \lfloor \sqrt[k]{m} \rfloor$. As with the previous case, $a+1-b^k \not\in\mathbf{B}^k_{j+1}$ by assumption, so there is again no positive $b, 1 \le b \le \lfloor \sqrt[k]{m} \rfloor$, such that $a+1-b^k \not\in\mathbf{B}^k_{j+1}$. However, $\max\left(\mathbf{B}^k_{j+1}\right) = m+1$, so it is now necessary to check values of $b$ up to $\lfloor\sqrt[k]{m+1}\rfloor$. Since $\lfloor\sqrt[k]{m+1}\rfloor = \lfloor\sqrt[k]{m}\rfloor$ by assumption, there are no other possible values of $b$, and so $a+2 \not\in\mathbf{B}^k_{j+2}$. Because $j$ is positive, $j+2 \ge 3$, and there is no way to write $1$ or $2$ as a sum of $j$ positive $k\th$ powers, completing the proof.
\end{proof}
\begin{theorem}\label{B_consistency} Let $m = \max\left(\mathbf{B}^k_{j}\right)$. Suppose $\mathbf{B}^k_{j+1} = $ \{1\} $\cup$ \{$n + 1 : n \in \mathbf{B}^k_j$\} and suppose that
$\left\lfloor\sqrt[k]{(m-j)(1+\frac{1}{2^k-1})}\right\rfloor =\left\lfloor\sqrt[k]{m}\right\rfloor$. Then, $\mathbf{B}^k = $ \{$n - j : n \in \mathbf{B}^k_j, n > j$\}, and $\overline{\mathbf{B}^k_l} = \emptyset$ for all $l \ge j$.
\end{theorem}
\begin{proof}
    By Lemma \ref{B_consistency_step3}, 
    \[
    \mathbf{B}^k_{j+2} =  \{1\} \cup \{n+1 : n \in \mathbf{B}^k_{j+1}\} = \{1,2\} \cup \{n+2: n \in \mathbf{B}^k_j\}. 
    \]
    Similarly, 
    \[\mathbf{B}^k_{j+3} =  \{1,2,3\} \cup \{n+3 : n \in \mathbf{B}^k_j\}, \mathbf{B}^k_{j+4} =  \{1,2,3,4\} \cup \{n+4 : n \in \mathbf{B}^k_j\}.
    \]
    This process continues inductively until 
    \[
    \mathbf{B}^k_{\left\lceil\frac{m-j}{2^k-1}\right\rceil} =  \{1,2,3,\ldots,\left\lceil\frac{m-j}{2^k-1}\right\rceil+1-j\} \cup \{n+\left\lceil\frac{m-j}{2^k-1}\right\rceil-j : n \in \mathbf{B}^k_j\}
    \] 
    of which the largest element is \[
    m+\left\lceil\frac{m-j}{2^k-1}\right\rceil-j < m-j + \frac{m-j}{2^k-1} + 1 = (m-j)\left(\frac{2^k}{2^k-1}\right) + 1 < 2^k\left(\left\lfloor\frac{m-j}{2^k-1}\right\rfloor + 1\right).
    \]
    By Lemma \ref{B_consistency_step2}, for any larger positive integer $l \geq j$, it follows that \[
    \mathbf{B}^k_{l+1} =  \{1\} \cup \{n + 1: n \in \mathbf{B}^k_l\}.
    \]
    Since the only difference between $\mathbf{B}^k_l$ and $\mathbf{B}^k_{l+1}$ is the addition of $1$ to each term of $\mathbf{B}^k_l$, it follows that the set  $\mathbf{B}^k = $ \{$n-j : n \in \mathbf{B}^k_j, n > j$\} by definition, and that $\overline{\mathbf{B}^k_l} = \emptyset$.
\end{proof}

For example, Grosswald's proof of $\mathbf{B}^2$ can be restated using the fact that $\mathbf{B}^2_6 = $ \{1, 2, 3, 4, 5\} $\cup$ \{7, 8, 10, 11, 13, 16, 19\} and $\mathbf{B}^2_7 = $ \{1, 2, 3, 4, 5, 6\} $\cup$ \{8, 9, 11, 12, 14, 17, 20\}. In this case, $m = 19, j = 6$, and $k = 2$. The second condition in Theorem \ref{B_consistency} becomes
\[ \left\lfloor\sqrt{13\left(1+\frac{1}{3}\right)}\right\rfloor = \left\lfloor\sqrt{19}\right\rfloor.\]

\subsection{The numbers $g(1,k)$ and $G(1,k)$}
Zenkin defines the sequence $\{g(1,k)\}_{k\ge1}$, where $g(1,k) := \min \{j:\  \overline{\B_j^k} = \emptyset\}$. The sequence $g(1,k)$ for $k=1,2,\ldots$ begins:
\[ g(1,k) = 1, 6, 15, 22, 58, 78, 244, 334, 717, \ldots
\]

The following table shows $g(1,k)$ in comparison with the other three sequences commonly cited in the generalized Waring problem: $g(k), G(k)$, and $G(1,k)$.

\begin{table}[ht!]
    \centering
    \begin{tabular}{c|ccccccccc}
         $k$& 1 & 2 & 3 & 4 & 5 & 6 & 7 & 8 & 9\\
         \hline
        $G(k)$ & 1 & 4 & 7* & 16 & 17* & 24* & 33* & 42* & 50*\\
        $G(1,k)$ & 1 & 5 & 9* & 18 & 20* & 29* & 40* & 52* & 117*\\
        $g(k)$ & 1 & 4 & 9 & 19 & 37 & 73 & 143 & 279 & 548\\
        $g(1,k)$ & 1 & 6 & 14 & 21 & 57 & 78 & 245 & 334 & 717
    \end{tabular}
    \caption{Waring's problem sequences for $1 \leq k \leq 9$}
    \label{Table_Sequence_|B^j|}
\end{table}
$^*$upper bound

\begin{conjecture}\label{g(1,k) increasing}
    The sequence $g(1,k)$ is strictly increasing.
\end{conjecture}
\begin{question}
    Is it possible to find an explicit formula or an asymptotic for the sequence $g(1,k)$?
\end{question}

\section{Proof of Remaining Theorems}
\begin{proof}[Proof of Theorems \ref{Theorem_2_fifths}, \& \ref{Theorem_3_fifths}]
    Because Vaughan \& Wooley \cite{Vaughan_Wooley_5ths} showed that $G(5)\le17$ for sufficiently large positive integers, there must be a largest such $n_0$ that is not the sum of $17$ non-negative fifth powers. (Conjecture \ref{Conjecture_n>87918 for fifths} predicts that the largest such positive integer is $87918$, but the actual value of $n_0$ is not relevant to the calculation of $G(1,5)$). Every $n > 100000497376 + n_0$ has a representation as the sum of exactly $20$ positive fifth powers, which proves that $G(1,5)=20$.
    
     The authors performed a numerical check to confirm that all $m$ such that $77529941 < m < 10^9$ are $(10,5)$-representable. By Theorem \ref{representation_existence}, using $a=117$, all $77529942 < m < 22924480357$ are $(11,5)$-representable. Applying Theorem \ref{representation_existence} again with $a=260$ yields the result that all $77529943 < m < 1211062080357$ (approximately $1.21\times10^{12}$ and greater than $100000497376 \approx 10^{11}$) are $(12,5)$-representable. Adding eight copies of $1^5$ to a $(12,5)$-representation shows that any $77529951 < m < 1.21\times10^{12}$ is $(20,5)$-representable. Assuming Conjecture \ref{Conjecture_n>87918 for fifths}, all $n > 87918$ are representable as the sum of at most $17$ positive fifth powers, so $b=17$ in Theorem \ref{n*_application}. Because the representations of $100000497376$ as sums of fifth powers start at $d=3$, all positive integers greater than $100000497376 + 87918 = 100000585294$ are $(b+d,5) = (20,5)$-representable. Hence, to check all values of $n \leq 100000585294$ are not the sum of $20$ positive fifth powers, it is only necessary to check values up to $77529951$, yielding the set $\mathbf{B}_{20}^5$:
    \[
    \mathbf{B}_{20}^5 = \{1, 2, \ldots, 19\} \cup \{20 + \beta \mid \beta \in \mathbf{B}^5\cup \overline{\mathbf{B}_{20}^5}\} .
    \]
    The complete set of $\overline{\mathbf{B}_{20}^5}$ is listed in the \nameref{Appendix} (assuming Conjecture \ref{Conjecture_n>87918 for fifths}). For $j = 21$, consider $n-1$ for $n > \max(\mathbf{B}_{20}^5) + 1$. Because $n-1$ is $(20,5)$-representable, then $n = \sum_{i=1}^{20}x_i^5+1^5$ is a $(21,5)$-representation of $n$. It is left to check for which values of $n \le \max(\mathbf{B}_{20}^5) + 1$ are not the sum of $21$ positive fifth powers:
    \[
    \mathbf{B}_{21}^5 = \{1, 2, \ldots, 20\} \cup \{21 + \beta \mid \beta \in \mathbf{B}^5 \cup \overline{\mathbf{B}_{21}^5}\}.
    \]
    The complete set $\overline{\mathbf{B}_{21}^5}$ is also listed in the \nameref{Appendix} (assuming Conjecture \ref{Conjecture_n>87918 for fifths}). This process may be continued indefinitely.
    To obtain an unconditional result, it is necessary to use $g(5)=37$, since $n_0$ is unknown. Since all $n > 0$ are representable as the sum of at most $37$ positive fifth powers, $b=37$ in Theorem \ref{n*_application}, while $d$ is still $3$. Thus, all positive integers greater than $100000497376$ are $(40,5)$-representable. Since it was established earlier that all $77529951 < m \le 100000497376$ are $(20,5)$-representable, adding $20$ copies of $1^5$ establishes that all $77529971 < m \le 100000497376$ are $(40,5)$-representable, making the remaining cases easy to check.
    
    The inclusion process continues inductively, where each $\overline{\mathbf{B}_{j}^5}$ is nonempty (and listed for completeness in the  \nameref{Appendix}), until $j=57$. It then becomes apparent that
    \[\{n-57 : n > 57, \in \mathbf{B}_{57}^5\} = \{n-58 : n > 58, \in \mathbf{B}_{58}^5\}.\]
    It appears that the set $\{n-57 : n > 57, \in \mathbf{B}_{57}^5\}$ never shrinks again upon replacing $57$ with some larger $j$, and is actually the set $\mathbf{B}^5$.
    
    In fact, $\mathbf{B}_{58}^5 = \{\beta+1 \mid \beta \in \mathbf{B}_{57}^5\} \cup \{1\}$. The largest value in $\mathbf{B}_{57}^5$ is $6318$, so by Theorem \ref{B_consistency}, where $m=6318, k=5$, and $j=57$, the first condition is satisfied, and the second condition evaluates to
    \[ \left\lfloor\sqrt[5]{6261\left(1+\frac{1}{31}\right)}\right\rfloor = \left\lfloor\sqrt[5]{6318}\right\rfloor.\]

    This proves that $\{n-57 : n > 57, \in \mathbf{B}_{57}^5\} = \mathbf{B}^5$.
\end{proof}

The proof technique above was adapted from Dubouis \cite{Dubouis} and the structure of the proof is essentially the same for each value of $k$. The task is to find the lowest value of $n^*$ possible and substitute the appropriate values. Following the same proof technique, but with appropriate substitutions will yield the remaining proofs. 

To prove Theorem \ref{Theorem_2_sixths}, let $n^*=41253168892$ which is the sum of $5, 6, 7, \ldots, 77$ positive sixth powers. To ensure that $\mathbf{B}^6 = $ \{$n - j : n \in \mathbf{B}^6_j, n > j$\}, and $\overline{\mathbf{B}^6_l} = \emptyset$ for all $l \ge j$, compute up to $j=78$, at which point the two conditions of Theorem \ref{B_consistency} are satisfied: $\mathbf{B}^6_{j+1} = \{1\} \cup \{n + 1 : n \in \mathbf{B}^6_j\}$ and $\left\lfloor\sqrt[6]{711649(1+\frac{1}{63})}\right\rfloor = \left\lfloor\sqrt[6]{711727}\right\rfloor$.

To prove Theorems \ref{Theorem_2_sevenths} and \ref{Theorem_3_sevenths}, let $n^*=822480142011$ which is the sum of $7, 8, 9 \ldots, 149$ positive seventh powers. To ensure that $\mathbf{B}^7 = $ \{$n - j : n \in \mathbf{B}^7_j, n > j$\}, and $\overline{\mathbf{B}^7_l} = \emptyset$ for all $l \ge j$, compute up to $j=245$, at which point the two conditions of Theorem \ref{B_consistency} are satisfied: $\mathbf{B}^7_{j+1} = \{1\} \cup \{n + 1 : n \in \mathbf{B}^7_j\}$ and $\left\lfloor\sqrt[7]{249077(1+\frac{1}{127})}\right\rfloor = \left\lfloor\sqrt[7]{249322}\right\rfloor$. 

To prove Theorems \ref{Theorem_2_eighths} and \ref{Theorem_3_eighths}, let $n^*=17373783550950$ which is the sum of $9, 10, 11 \ldots, 287$ positive eighth powers. To ensure that $\mathbf{B}^8 = $ \{$n - j : n \in \mathbf{B}^8_j, n > j$\}, and $\overline{\mathbf{B}^8_l} = \emptyset$ for all $l \ge j$, compute up to $j=334$, at which point the two conditions of Theorem \ref{B_consistency} are satisfied: $\mathbf{B}^8_{j+1} = \{1\} \cup \{n + 1 : n \in \mathbf{B}^8_j\}$ and $\left\lfloor\sqrt[8]{1890241(1+\frac{1}{255})}\right\rfloor = \left\lfloor\sqrt[8]{1890575}\right\rfloor$. 

To prove Theorems \ref{Theorem_2_ninths} and \ref{Theorem_3_ninths}, let $n^*=25636699123453928$ which is the sum of $14, 15, 16 \ldots, 561$ positive ninth powers. To ensure that $\mathbf{B}^9 = $ \{$n - j : n \in \mathbf{B}^9_j, n > j$\}, and $\overline{\mathbf{B}^9_l} = \emptyset$ for all $l \ge j$, compute up to $j=717$, at which point the two conditions of Theorem \ref{B_consistency} are satisfied: $\mathbf{B}^9_{j+1} = \{1\} \cup \{n + 1 : n \in \mathbf{B}^9_j\}$ and $\left\lfloor\sqrt[9]{6464397(1+\frac{1}{511})}\right\rfloor = \left\lfloor\sqrt[9]{6465114}\right\rfloor$. 

 Finally, proof of Theorems \ref{theorem_3_cubes} and \ref{theorem_fourths} follow from direct computation. Zenkin knew in the early 1990's that in the setting of perfect cubes, $n^*=1072$ was optimal ($d=2$), and computed $\mathbf{B}^3$, but it would require subsequent work to eventually prove in 2016 a conjecture made by Jacobi in 1851 \cite{Linnik, Watson, Mccurley, Ramare, Bertault, Boklan_Elkies, ekl1998new, Siksek}. 
\begin{lemma}[Jacobi-Siksek]\label{Siksek}
    Every positive integer other than 
    \[
    15, 22, 23, 50, 114, 167, 175, 186, 212, 231, 238, 239, 303, 364, 420, 428, 454
    \]
    can be expressed as the sum of at most $7$ positive cubes.
\end{lemma}
\noindent At this point, the proof technique can be repeated as before and the sets $\overline{\mathbf{B}_j^3}$ for $j=13, 12, 11, 10, 9$ are confirmed. 

Similarly, Zenkin knew that in the setting of fourth powers, $n^*=77900162$ is optimal $(d=2)$, and computed $\mathbf{B}^4$. In 2000, Deshouillers, Hennecart, \& Landreau \cite{Deshouillers2000} computed exactly which positive integers cannot be expressed as the sum of at most $16$ nonnegative fourth powers. 
\begin{lemma}[Deshouillers, Hennecart, \& Landreau]\label{DHL}
    Every positive integer other than 
    \begin{align*}
    &47, 62, 63, 77, 78, 79, 127, 142, 143, 157, 158, 159, 207, 222, 223, 237, 238, 239, 287, 302, 303, 317, 318,\\
    &319, 367, 382, 383, 397, 398, 399, 447, 462, 463, 477, 478, 479, 527, 542, 543, 557, 558, 559, 607, 622,\\
    &623, 687, 702, 703, 752, 767, 782, 783, 847, 862, 863, 927, 942, 943, 992, 1007, 1008, 1022, 1023, 1087,\\
    &1102, 1103, 1167, 1182, 1183, 1232, 1247, 1248, 1327, 1407, 1487, 1567, 1647, 1727, 1807, 2032, 2272,\\
    &2544, 3552, 3568, 3727, 3792, 3808, 4592, 4832, 6128, 6352, 6368, 7152, 8672, 10992, 13792
    \end{align*}
    can be expressed as the sum of $16$ nonnegative fourth powers.
\end{lemma}
The proof technique can again be repeated to confirm the sets $\overline{\mathbf{B}_j^4}$ for $j=20, 19, 18$.

\section{Future Improvements}\label{Improvements}

Note that this proof technique suffices to establish the sets $\mathbf{B}_j^k$ for $j\geq g(k)+d$, and establish that the sets $\mathbf{B}_j^k$ for $G(k)+d < j < g(k)+d$ are finite. Finding lower values of $d$ would improve the results in this paper. Alternatively, improvements in Waring's problem would also improve the results of this paper, with substitution of a lower value of $G(k)$. The number $100000585294$ works in the proof of Theorems \ref{Theorem_2_fifths}, \& \ref{Theorem_3_fifths} because it is small enough for modest computers to find consecutive expressions as the sum of $3, 4, 5, \ldots, 40$ positive fifth powers, but there are larger numbers that may slightly improve the results of Theorems \ref{Theorem_2_fifths}, \& \ref{Theorem_3_fifths}. Fermat's Last Theorem guarantees no perfect fifth power is the sum of two fifth powers, but, for example, the smallest positive integer that is the sum of $2$ and $3$ positive fifth powers was computed by Scher \& Seidl \cite{ekl1998new}:
\[
563661204304422162432 = 14132^5+220^5 = 14068^5+6237^5+5027^5.
\]
How far could this technique work for fifths (and higher powers)? Finding a positive integer that is the sum of $2, 3, 4, \ldots, 40$ positive fifth powers would prove which numbers are not the sum of $j \geq 19$ positive fifth powers. This is the theoretical limit that this proof technique will allow for all positive powers without further improvements in Waring's problem.\\

The amount of numbers less than $N$ expressible as a sum of two fifth powers is bounded above by the number of lattice points bounded by the $x-$ axis, the curve $x^5+y^5=N$, and the line $y=x$, which is asymptotically equivalent of the bounded area $A$. This is also equivalent to half of the area bounded by the $x-$ and $y-$ axes and the curve $x^5+y^5=N.$ Furthermore, scaling down the figure by a factor of $\sqrt[5]{N}$ changes the curve to $x^5+y^5=1.$ The $x-$ intercept is at $(1,0),$ and the equation can be rewritten as $y = \sqrt[5]{1-x^5},$ giving a formula for the area $A'$ of the scaled-down figure:
\[ A' = \int_0^1 \sqrt[5]{1-x^5} dx.\]

Numerical integration shows that $0.9501 < A' < 0.9502,$ and $A' \approx 0.95015$.
Since the area in this case is two-dimensional, $A = \left(\sqrt[5]{N}\right)^2A' = N^\frac{2}{5}A' \approx 0.95015N^\frac{2}{5},$ and halving it gives $0.47508N^\frac{2}{5}.$ This is an asymptotic upper bound for the amount of numbers that are the sum of two fifth powers, since there may be some numbers that are the sum of two fifth powers in two or more ways (although this is, heuristically, highly unlikely). Taking the derivative of the sum, from 1 to $N$, of the asymptotic expected number of sums of two fifth powers, and evaluating at $N$ yields the probability that $N$ is expressible as a sum of two fifth powers:
\[ P_2(N) := \frac{dA}{dN} \approx 0.19003N^{-\frac{3}{5}}.\]

Similarly, to find the probability that a number is the sum of three fifth powers, let $A_3'$ be the volume bounded by the axes and the surface $x^5+y^5+z^5=1$. Define $A_3$ as the volume bounded by the $x$-axis, the lines $y=x,$ $z=y$ (such that $x \ge y \ge z$), and the surface $x^5+y^5+z^5=N$. Rearranging the first surface gives $y^5+z^5=1-x^5.$ Taking cross-sections of $A_3'$ with $x$ fixed yields a formula for its volume in terms of the earlier calculated result for the area under the curve $x^5+y^5=N$, with $N = 1-x^5$: 
\[ A_3' = \int_0^1 0.95015\left(1-x^5\right)^\frac{2}{5} dx \approx 0.86629.\]
Since the equation is symmetric, without loss of generality, let $x>y>z$. This reduces the number of integer points, and thus volume, by a factor of $3!$. Hence, 
\[A_3=\frac{1}{3!}N^\frac{3}{5}A_3' \approx 0.14438N^\frac{3}{5} \qquad P_3(N) := \frac{dA_3}{dN} = 0.08663N^{-\frac{2}{5}}.\]
The expected number of solutions to $v^5+w^5=x^5+y^5+z^5$ across the positive integers is
\[ \int_1^\infty P2(x)P3(x) dx = \int_1^\infty 0.01646x^{-1} dx = \mid_1^\infty 0.01646\ln(x) = \infty.\]
However, none of these solutions are likely to also have representations as a sum of four fifth powers. Given that $563661204304422162432$ is the smallest solution, the probability that it is the sum of four fifth powers can be calculated, along with the probability that any larger number is the sum of $2, 3,$ and $4$ fifth powers. Let $A_4'$ be the volume bounded by the axes and $w^5+x^5+y^5+z^5=1,$ or $x^5+y^5+z^5=1-w^5.$ Then,
\[A_4' = \int_0^1 6A_3(1-w^5) dw = \int_0^1 0.86629(1-w^5)^\frac{3}{5} dw = 0.76306.\]
There are $4!$ ways to order the variables; but it is only necessary to count those with $w\ge x\ge y\ge z.$ Thus,
\[A_4 = \frac{1}{4!}A_4'N^\frac{4}{5} = 0.03179N^\frac{4}{5} \qquad P_4(N) := \frac{dA_4}{dN} = 0.02544N^{-\frac{1}{5}}.\]
Plugging in 563661204304422162432 shows that the probability $p_1$ it is the sum of four positive fifth powers is 
\begin{align}\label{p1}
    P4(563661204304422162432) = 0.02544\cdot563661204304422162432^{-\frac{1}{5}} = 0.0000018.
\end{align} The probability that any larger number is the sum of 2, 3, and 4 positive fifth powers is
\begin{align}\label{p2}
\nonumber \int_{563661204304422162432}^\infty P_2(x)P_3(x)P_4(x) dx = & \int_{563661204304422162432}^\infty 0.0004188x^{-\frac{1}{5}} dx \\
\nonumber = & \ (0.00008376)\left(563661204304422162432\right)^{-\frac{1}{5}} \\ \approx & \ 9.39362\times10^{-9}. 
\end{align}
Adding \eqref{p1} and \eqref{p2} gives a final probability of 0.000001809 for the existence of a number that is $(2,5)$-representable, $(3,5)$-representable, and $(4,5)$-representable.

In general, the number of $(j,k)$-representations less than $n$ is asymptotically \[
P(n,j,k)\approx\frac{n^{j/k}}{j!}V(j,k)
\]
where $V(j,k)$ is the $j$-dimensional volume between the axes and the $j$-surface $x_1^k + \cdots + x_j^k = 1$. Consequently, taking the derivative of this counting function gives $P'(n,j,k)=\frac{n^{j/k-1}}{k(j-1)!}$ as the proportion of numbers near $n$ that are $(j,k)$-representable. If $P'(n,j,k)$ is small and there are few numbers with multiple representations, this is also approximately the heuristic probability that $n$ is $(j,k)$-representable.

For small $j$ and $k$, $V(j,k)$ can be calculated using an integral, although for large cases it may be more effective to use a Monte Carlo simulation. Heuristically, the ``probability" (assuming $(j,k)$-representations are randomly distributed) of $n$ being $(j_1,k_1), \ldots (j_m,k_m)$-representable is approximately 
\begin{align}\label{abcd}
C \cdot n^{j_1/k_1 + \cdots + j_m/k_m - m},
\end{align}
where $C$ is a constant. Integrating \eqref{abcd} from a lower bound $b$ for any solution to infinity gives the expected value for the quantity of such numbers over the positive integers:
\[\int_b^{\infty}C \cdot n^{j_1/k_1 + \cdots + j_m/k_m - m} = |_{b}^{\infty}C \cdot n^{j_1/k_1 + \cdots + j_n/k_n - n+1}.\] Since $C$ is usually smaller for larger $j_i$, $k_i$, and $m$, the existence of solutions increasingly depends on the exponent $E = j_1/k_1 + \cdots + j_m/k_m - m$. If $E < -1$, the upper bound of the integral is zero, leaving $C \cdot b^{j_1/k_1 + \cdots + j_m/k_m - m+1}$. If $E = -1$, the integration shown above is invalid and should yield a logarithm instead. The proper evaluation of the integral indicates there are infinitely many solutions, although sometimes there are none (as in the $n=3$ case of Fermat's Last Theorem). If $E > -1$, the upper bound of the integration is infinite. These observations make it possible to predict the minimum value of $d$ (from Table \ref{Table_Powers}) that will be successful. If the constants denoted here as $C$ are carried throughout the integration, setting the integral to 1 gives an approximation for how large the smallest solution $n^*$ will be. 

For example, when $k_1,\ldots,k_m=6$, if $d=3$, then $E = \frac{3}{6}+\frac{4}{6}+\frac{5}{6}-3=-1$ (exponents greater than 5 are omitted because representations eventually become so frequent that they break the probability model), whereas, if $d=2$, then $E = -\frac{5}{3}$. Solutions where $d=4$ are certainly expected, because $E = -\frac{1}{2} > -1$, but may be quite large.

Suppose a solution exists starting at a specified value $d$. Using all $j < k$, \[E = \sum_{i=d}^{k-1} \left(\frac{i}{k}-1\right) = \sum_{i=1}^{k-d} \frac{i}{k} = -\frac{(k-d)(k-d+1)}{2k} = -1.\] Hence, $(k-d)(k-d+1) = 2k$. Setting $x=k-d$ gives $x^2+x=2k$ or $x = \frac{-1\pm\sqrt{1+8k}}{2}$. Since the negative solution is absurd, as it makes the summation formula for $E$ an empty sum, the minimum value of $d$ expected to yield a solution is $d = k + \frac{1-\sqrt{1+8k}}{2}$. Note when the formula for $d$ yields an integer, the solution produced yields $E=-1$, and a solution may or may not exist. When $k=3$, this formula yields $d=1$, which violates Fermat's Last Theorem.

While it may be impractical to reduce $d$ for now, the results of this paper could also be improved by theoretical advancements in the classical Waring problem. Additionally, it appears that many more $\mathbf{B}_j^k$ (and, by extension, $\overline{\mathbf{B}_j^k}$) are finite than can be proven using this technique. To be able to completely determine more sets $\overline{\mathbf{B}_j^k}$, it would be helpful to know the largest number that is not $(j,k)$-representable for various $j$ and $k$. To this end, consider the following conjectures:
\begin{conjecture}
    Every positive integer greater than $77529941$ is the sum of exactly $10$ positive fifth powers.
\end{conjecture}
\begin{conjecture}
Every positive integer greater than $229182966$ is the sum of $18$ positive sixth powers.
    


\end{conjecture}

\begin{conjecture}
    Every positive integer greater than $317476671$ is the sum of exactly $25$ positive seventh powers.
\end{conjecture}

\begin{conjecture}
    Every positive integer greater than $904339959$ is the sum of exactly $47$ positive eighth powers.
\end{conjecture}

\begin{conjecture}\label{conj for 9ths}
   Every positive integer greater than $967214052$ is the sum of exactly $121$ positive ninth powers.
\end{conjecture}



\bibliographystyle{plain}
\bibliography{PositivePowers}{}

\section*{Appendix}\label{Appendix}
The Appendix for this paper can be found by clicking or scanning the qr-code below:
\begin{center}
   \qrcode{http://brennanbenfield.com/appendix}
\end{center}
\end{document}